\documentclass{amsart}
\usepackage{a4wide}
\usepackage[usenames]{color}

\newtheorem{lemma}{Lemma}[section]
\newtheorem{thm}[lemma]{Theorem}
\newtheorem{definition}[lemma]{Definition}

\newtheorem{corollary}[lemma]{Corollary}

\title{A Note on Symmetries in the Rauzy graph and Factor Frequencies}

\author{L\!'ubom\'ira Balkov\'a}
\author{Edita Pelantov\'a}
\address{Doppler Institute for Mathematical Physics and Applied Mathematics,
and Department of Mathematics, FNSPE, Czech Technical University,
Trojanova 13, 120~00 Praha~2, Czech Republic}
\email{l.balkova@centrum.cz, Edita.Pelantova@fjfi.cvut.cz}

\begin{document}

\begin{abstract}
We focus on infinite words with languages closed
under reversal. If frequencies of all factors are well defined,  we
show that the number of different frequencies of factors of
length $n+1$ does not exceed  $2\Delta C (n)+1$,
where $\Delta C (n)$  is the first difference of factor complexity
$C(n)$ of the infinite word.

\end{abstract}
\maketitle

\section{Introduction}

It is well-known that the Rauzy graph, despite of its simplicity,
has turned out to be a~powerful tool in the study of various
combinatorial properties of words. The first one to
use the idea to label edges of the Rauzy graph with frequencies was
Dekking~\cite{De} in order to show that for every length, there
exists at most three different factor frequencies in the Fibonacci
sequence. Moreover, he described for every length $n$, the set of
frequencies of factors of length $n$ and the number of factors of
length $n$ having the same frequency. Berth\'e in~\cite{Be},
observing also the evolution of Rauzy graphs for growing factor
lengths, generalized Dekking's result for all Sturmian
words.\footnote{Note that this result follows also from the $3$ gap
theorem, see \cite{So}.}

With help of the Rauzy graph,
Boshernitzan~\cite{Bo} deduced an upper bound on the number of different
frequencies in a general recurrent infinite word. He showed that the
number of frequencies of factors of length $n+1$ does not exceed
$3\Delta C(n)$, where $\Delta C (n)$ is the first difference of
factor complexity of the infinite word.

 Since $\Delta C(n)$ is known to be bounded
  for infinite words with sublinear complexity (see~\cite{Ca0}), it implies
   for fixed points of primitive substitutions and for fixed points of uniform
    substitutions (all images of letters have the same length) that the number
     of different frequencies of factors of the same length is bounded.

Boshernitzan's upper bound  $3\Delta C(n)$ can be
further diminished, if the labeled Rauzy graphs corresponding to an
infinite word have a~nontrivial group of automorphisms. This
property of the Rauzy graphs is guaranteed for example if the
language of an infinite word is closed under reversal or closed
under permutation of letters. The main aim of this paper is to prove
the following theorem:

\begin{thm}\label{UpperBoundReversalClosed} Let $u$ be an infinite word whose
language is closed under reversal and such that  the frequency
$\rho(w)$ exists for every factor $w$ of the word $u$.  Then for
every $n \in \mathbb N$, we have
\begin{equation}\label{odhad}
\# \{\rho(w)| w \in {\mathcal L}_{n+1} \}\quad \leq \quad 2\Delta
C(n)+1,
\end{equation}
 where ${\mathcal
L}_{n+1}$ denotes the set of factors of $u$ of length $n+1$.
\end{thm}

We also deduce  that the equality  holds for all sufficiently large
$n$ if and only if $u$ is periodic. Nevertheless, a~recent result of
Ferenczi and Zamboni shows that this bound cannot be improved,
keeping its general validity, even for aperiodic words whose
languages are closed under reversal. In~\cite{FeZa}, they study the
infinite words coding $k$-interval exchange transformation with the
symmetric permutation. The authors show among others that for such
infinite words, the equality in
Theorem~\ref{UpperBoundReversalClosed} is reached infinitely many
times. (In fact, they proved a~stronger statement: the set of
indices $n$ for which the equality~\eqref{odhad} holds has density
one.)

Finally, let us mention that the idea to exploit a~symmetry of the
Rauzy graph was already used in~\cite{BalMaPe} in order to estimate
the number of palindromes of a~given length. Our article is intended
as a~further example why it is useful to study symmetries in Rauzy
graphs.

\section{Preliminaries}
An {\em alphabet} $\mathcal A$ is a~finite set of symbols, called
{\em letters}. A~concatenation of letters is a~{\em word}. {\em
Length} of a~word $w$ is the number of letters contained in $w$ and
is denoted $|w|$. The set $\mathcal A^{*}$ of all finite words
(including the empty word $\varepsilon$) provided with the operation
of concatenation is a~free monoid. We will also deal with
right-sided infinite words $u=u_0u_1u_2...$. A~finite word $w$ is
called a~{\em factor} of the word $u$ (finite or infinite) if there
exist a~finite word $w^{(1)}$ and a~word $w^{(2)}$ (finite or
infinite) such that $u=w^{(1)}ww^{(2)}$. The factor $w^{(1)}$ is
a~{\em prefix} of $u$ and $w^{(2)}$ is a~suffix of $u$. An infinite
word $u$ is said to be {\em recurrent} if each of its factors occur
infinitely many times in $u$.

{\em Language} ${\mathcal L}$ of an infinite word $u$ is the set of
all factors of $u$.  We denote by ${\mathcal L}_n$ the set of
factors of length $n$ of the infinite word $u$. Then, we can define
{\em complexity function} (or {\em complexity}) $C: \mathbb N
\rightarrow \mathbb N$ which associates to every $n$ the number of
different factors of length $n$ of the infinite word $u$, i.e.
$C(n)=\# {\mathcal L}_n.$

An important role for determining the factor complexity is played by
special factors. We say that a~letter $a$ is {\em right extension}
of a~factor $w \in {\mathcal L}$ if $wa$ is also a~factor of $u$. We
denote by $Rext(w)$ the set of all right extensions of $w$ in $u$,
i.e. $Rext(w)=\{a \in {\mathcal A} \bigm | wa \in {\mathcal L}\}$.
If $\#Rext(w)\geq 2$, then the factor $w$ is called {\em right
special} (RS for short). Analogously, we define {\em left
extensions, $Lext(w)$, left special factor} (LS for short).
Moreover, we say that a~factor $w$ is {\em bispecial} (BS for short)
if $w$ is LS and RS.

With this in hand, we can introduce a~formula
for the {\em first difference of complexity} $\Delta
C(n)=C(n+1)-C(n)$ (taken from~\cite{Ca1}).
\begin{equation} \label{complexity}
\Delta C(n)=\sum_{w \in  {\mathcal L}_n}\bigl (\# Rext(w)-1 \bigr
)=\sum_{w \in  {\mathcal L}_n}\bigl (\# Lext(w)-1 \bigr ), \quad n
\in \mathbb N.
\end{equation}

A~language ${\mathcal L}$ is closed under reversal, if for every
factor $w=w_1\dots w_n\in \mathcal A^{*}$ also its mirror image
$\overline{w}=w_n\dots w_1$  belongs to ${\mathcal L}$. A factor $w$
which coincides with its mirror image $\overline{w}$ is called {\em
palindrome}.

If we  denote by ${\mathcal Pal}_n$ the set of palindromes of length
$n$ contained in $u$, then we can define
 {\em palindromic complexity} $P: \mathbb N
\rightarrow \mathbb N$ of the infinite word $u$ by the prescription
$P(n)=\# {\mathcal Pal}_n$. Clearly, $P(n)\leq C(n)$ for any
positive integer $n$.
 A non-trivial inequality between  $P(n)$ and $C(n)$  can be
found in  \cite{AlBaCaDa}. Here we shall use the result from
\cite{BalMaPe}: if the language of  an infinite recurrent word is
closed under reversal, then
 \begin{equation}\label{odhadPal}
P(n)+P(n+1)  \ \leq \ \Delta C(n) +2.
 \end{equation}
In this paper, we focus on infinite words with well defined factor
frequencies. More precisely, we will assume that for
any factor $w$ of an infinite word $u$, the following limit exists
$$\lim_{|v| \to \infty, v \in {\mathcal L}}
\frac {\# \{ \mbox{occurrences of $w$ in $v$} \} }{|v|}.$$ This
limit will be denoted by $\rho(w)$ and called {\em frequency} of the
factor $w$. Let us add that an {\em occurrence} of $w$ in
$v=v_1v_2\ldots v_m$ is an index $i \leq m$ such that $w$ is
a~prefix of the word $v_iv_{i+1} \ldots v_m$.

To dispose of all definitions needed for the
deduction of an improved upper bound on the number of different
frequencies, it remains to define the labeled Rauzy graph.

{\em Labeled Rauzy graph} of order $n$ of an infinite word $u$ is
a~directed graph  $\Gamma_n$ whose set of vertices is ${\mathcal
L}_n$ and  set of edges is ${\mathcal L}_{n+1}$. Any edge $e = w_0
w_1 \dots w_n$ starts in the vertex $w=w_0w_1\dots w_{n-1}$, ends in
the vertex  $v=w_1\dots w_{n-1}w_n$, and is labeled by its factor
frequency $\rho(e)$.

\section{Reduced Rauzy graphs}

Edge frequencies in a~Rauzy graph $\Gamma_n$ behave similarly as
current in a~circuit.We may formulate an analogy of Kirchhoff's law:
the sum of frequencies of edges ending in a~vertex equals the sum of
frequencies of edges starting in this vertex. As a~direct
consequence, if a~Rauzy graph contains a~vertex with only one
incoming and one outgoing edge, then the frequency of these edges is
the same, say $\rho$. Therefore, we can replace this triple
(edge-vertex-edge) with only one edge keeping the frequency $\rho$.
If we reduce the Rauzy graph step by step applying the above
described procedure, we obtain the so-called {\em reduced Rauzy
graph} $\tilde{\Gamma}_n$, which simplifies the investigation of
edge frequencies. In order to precise this consideration, we
introduce the following notion.
 \begin{definition}\label{simple_path}
  Let
$\Gamma_n$ be the labeled Rauzy graph of order $n$
of an infinite word $u$. A~directed path $w^{(0)}w^{(1)}\dots
w^{(m)}$ of non-zero length in $\Gamma_n$ such that its initial
vertex $w^{(0)}$ and its final vertex $w^{(m)}$ are LS
or RS, and the other vertices are neither LS nor RS factors is
called simple. We define label of the simple path as the label of
any edge of this path.
 \end{definition}

 \begin{definition}\label{reduced_Rauzy_graph}
 Reduced Rauzy graph
$\tilde{\Gamma}_n$ of $u$ (of order $n$) is a~directed graph whose
set of vertices is formed by LS and RS factors of ${\mathcal L}_n$
and whose set of edges is given in the following way. Vertices $w$
and $v$ are connected with an edge $e$ if there exists in $\Gamma_n$
a~simple path starting in $w$ and ending in $v$. We assign to such
an edge $e$ the label of the corresponding simple path.
  \end{definition}
For a~recurrent word $u$, at least one edge starts and at least one
edge ends in every vertex of $\Gamma_n$. Therefore, no edge label is
lost by the reduction of $\Gamma_n$. The number of different edge
labels in the reduced Rauzy graph $\tilde{\Gamma}_n$ is clearly less
or equal to the number of edges in $\tilde{\Gamma}_n$. Let us thus
calculate the number of edges in $\tilde{\Gamma}_n$ in order to get
an upper bound on the number of frequencies of factors in $
{\mathcal L}_{n+1}$.

For every RS factor $w \in {\mathcal L}_n$, it holds that
$\#Rext(w)$ edges begin in $w$, and for every LS factor $v \in
{\mathcal L}_n$ which is not RS, only one edge begins in $v$, thus
we get the following relation
\begin{equation} \label{frequencies_vertices} \# \{ e
 | \ e \ \mbox{edge in} \ \tilde {\Gamma}_n\}=\sum_{w \ \text{RS in ${\mathcal L}_n$}} \# Rext(w)+\sum_{v \
\text{LS} \ \text{not RS in ${\mathcal L}_n$}}1. \end{equation}
Using Equation~(\ref{complexity}), we deduce that
\begin{equation}\label{edge_vertex} \# \{ e | \ e \ \mbox{edge in} \ \tilde {\Gamma}_n\}=\Delta C(n)+
\sum_{v \ \text{RS in ${\mathcal L}_n$} }1+\sum_{v \ \text{LS} \
\text{not RS in ${\mathcal L}_n$}}1.
\end{equation}
 Since $\# Rext(w) - 1 \geq 1$
for any RS factor $w$ and, similarly, for LS factors, we have
\begin{equation}\label{odhadRS}
\#\{w\in {\mathcal L}_n | \ w\ RS\} \ \leq  \ \ \Delta C(n) \quad
{\rm and} \quad \#\{w\in {\mathcal L}_n | \  w\ LS\} \ \leq  \ \
\Delta C(n)
\end{equation}
 The following result initially proved by Boshernitzan in~\cite{Bo}
 follows immediately by combining \eqref{edge_vertex} and \eqref{odhadRS}.
\begin{thm}\label{Boshernitzan} Let $u$ be an infinite recurrent word
such that for every factor $w \in {\mathcal L}$, the frequency
$\rho(w)$ exists. Then for every $n \in \mathbb N$, it holds
\[\#\{\rho(e) \bigm | e \in {\mathcal L}_{n+1}\}\quad \leq \quad 3\Delta
C(n).\] \end{thm}

\section{Proof of the Theorem  \ref{UpperBoundReversalClosed}}
Let us focus in the sequel on infinite words $u$ whose languages are
closed under reversal and such that the frequency of every factor
exists.
\begin{enumerate}
\item Such words are necessarily recurrent.
\item For any pair of factors $w, v \in {\mathcal L}$, it holds
$$\frac {\# \{ \mbox{occurrences of $w$ in $v$} \} }{|v|}=\frac
{\# \{ \mbox{occurrences of $\overline{w}$ in $\overline{v}$} \}
}{|\overline{v}|}.$$ Consequently, $\rho(w)=\rho(\overline{w})$ for
all factors $w$ of $u$.
\end{enumerate}
With the above two ingredients in hand, we will be
able to prove an essential lemma. Proof of
Theorem~\ref{UpperBoundReversalClosed} will be then a~direct
consequence of this lemma.
\begin{lemma}\label{technical_frequency} Let $u$ be an infinite word
whose language ${\mathcal L}$ is closed under reversal and such that
for each factor $w \in {\mathcal L}$, the frequency $\rho(w)$
exists. Then for every $n \in \mathbb N$, we have \[\# \{\rho(e)| e
\in {\mathcal L}_{n+1} \}\quad \leq \quad \frac{1}{2}\
\Bigl(P(n)+P(n+1)+\Delta C(n)- X-Y\Bigr)+Z,\]
\begin{tabular}{ll}
where & $X$ is the number of BS factors of length $n$, \\ & $Y$ is
the number of BS palindromic factors of length $n$, \\ & $Z$ is the
number of RS factors of length $n$.
\end{tabular}
 \end{lemma}
\begin{proof} Let $\Gamma_n$ be the labeled Rauzy graph of $u$ of order $n$.
Let us define a~mapping $\mu$ which to every vertex $w \in {\mathcal
L}_n$ associates the vertex $\overline{w}$, to every edge $e \in
{\mathcal L}_{n+1}$ associates the edge $\overline{e}$. Then,
${\mu}^{2}=Id$, and, thanks to the closeness of ${\mathcal L}$ under
reversal, $\mu$ maps $\Gamma_n$ onto itself, in fact, $\mu$ is an
automorphism of $\Gamma_n$. Clearly, every simple path
$w^{(0)}w^{(1)} \dots w^{(m)}$ in $\Gamma_n$ is mapped by $\mu$ to
the simple path $\overline{w^{(m)}}\dots \overline{w^{(1)}} \
\overline{w^{(0)}}$. This implies that $\mu$ induces an automorphism
on the reduced Rauzy graph $\tilde{\Gamma}_n$, too.

We know already that the set of edge labels of $\tilde{\Gamma}_n$ is
equal to the set of edge labels of $\Gamma_n$. Let us denote by $A$
the number of edges $e$ in $\tilde{\Gamma}_n$ (the number of simple
paths in $\Gamma_n$) such that $e$ is mapped by $\mu$ onto itself
and by $B$ the number of edges $e$ in $\tilde{\Gamma}_n$ such that
$e$ is not mapped by $\mu$ onto itself, then clearly, \[ \# \{ e | \
e \ \mbox{edge in} \ \tilde {\Gamma}_n\}=A+B.\] If $e$ is mapped by
$\mu$ onto itself, then the corresponding simple path satisfies
$$w^{(0)}w^{(1)}\dots w^{(m)}=\overline{w^{(m)}}\dots
\overline{w^{(1)}} \ \overline{w^{(0)}},$$ hence, for $m$ even, its
central vertex $w^{(\frac{m}{2})}$ is a~palindrome, and for $m$ odd,
its central edge going from $w^{(\frac{m-1}{2})}$ to
$w^{(\frac{m+1}{2})}$ is a~palindrome. On the other hand, every
palindrome of length $n+1$ is the central factor of a~simple path
mapped by $\mu$ onto itself and every palindrome of length $n$ is
either the central vertex of a~simple path mapped by $\mu$ onto
itself or is BS. Therefore,
\begin{equation}\label{A} A=P(n)+P(n+1)-\#\{w \in {\mathcal L}_n | w
\ \mbox{BS in ${\mathcal Pal}_n$} \}. \end{equation} We subtract the
number of palindromic BS factors of ${\mathcal L}_n$, in the
statement denoted by $Y$, since they are not inner vertices of any
simple path.

Now, let us turn our attention to edges of $\tilde{\Gamma}_n$ which
are not mapped by $\mu$ onto themselves. For every such edge $e$, at
least one another edge, namely $\mu(e)$, has the same label
$\rho(e)$. These considerations lead to the following estimate
\begin{equation}\label{A_AB} \# \{  \rho(e) | \ e \in {\mathcal L}_{n+1}\} \leq
A+\tfrac{1}{2}B=\tfrac{1}{2}A+\tfrac{1}{2}(A+B). \end{equation}
Rewriting Equation~(\ref{edge_vertex}), we obtain
\[A+B=\Delta C(n)+2Z-X.\] This fact together with~(\ref{A}) and \eqref{A_AB}
proves the statement. \end{proof} If we apply on $P(n)+P(n+1)$ and
$Z$ from Lemma~\ref{technical_frequency} the
estimates~\eqref{odhadPal} and~\eqref{odhadRS}, respectively, we
obtain immediately Proof of Theorem~\ref{UpperBoundReversalClosed}.
In fact, we get even a~finer upper bound
\begin{equation}\label{BetterUpBound} \# \{\rho(e)| e \in {\mathcal
L}_{n+1} \}\quad \leq \quad 2\Delta C(n)+1-
\tfrac{1}{2}X-\tfrac{1}{2}Y,
\end{equation}
where $X$ is the number of BS factors of length $n$ and $Y$ is
the number of BS palindromic factors of length $n$.

Let us study for which infinite words, the equality in
Theorem~\ref{UpperBoundReversalClosed} is attained. Infinite words
whose languages are closed under reversal are either purely periodic
or aperiodic.
\begin{itemize}
\item In case of purely periodic words, for sufficiently large $n$,
the first difference of complexity $\Delta C(n)=0$ and all factors
of length $n$ have the same frequency.
\item
On the other hand, aperiodic words contain infinitely many BS
factors. Hence, according to~\eqref{BetterUpBound}, the inequality
in Theorem~\ref{UpperBoundReversalClosed} is strict for infinitely
many $n$.
\end{itemize}
This reasoning leads to the following corollary.
\begin{corollary}
 Let $u$ be an infinite word
whose language ${\mathcal L}$ is closed under reversal and such
that for each factor $w \in {\mathcal L}$, the frequency
$\rho(w)$ exists. Then, the equality $$\# \{\rho(e)| e \in {\mathcal
L}_{n+1} \}\quad = \quad 2\Delta C(n)+1$$ holds for all sufficiently
large $n$ if and only if $u$ is periodic.
\end{corollary}
\section{Comments}
\begin{enumerate}
\item Berth\'e in~\cite{Be} has shown that for every Sturmian word, the number of frequencies of factors of length $n$ equals 2 if ${\mathcal L}_n$ contains a~BS factor, and is equal to 3 otherwise.
Since any BS factor of a~Sturmian word is a~palindrome, the finer
upper bound in~\eqref{BetterUpBound}  is reached for all $n \in
\mathbb N$.
\item Ferenczi and Zamboni~\cite{FeZa} have proved that infinite words coding $k$-interval exchange transformation whose language is closed under reversal attain the upper bound in~\eqref{BetterUpBound} for all $n \in \mathbb N$. As Sturmian words are infinite words coding 2-interval exchange transformation, Item (1) is a~particular case of their result.
\item Another example of infinite words for which the upper bound in~Theorem~\ref{UpperBoundReversalClosed} is reached infinitely many times are fixed points of the following substitution $\varphi$ on $\{0,1\}$:
$$\varphi(0)=0^{a}1, \quad \varphi(1)=0^b1, \quad a>b \geq 1.$$
The substitution $\varphi$ is a~canonical substitution associated with quadratic non-simple Parry numbers (for the precise definition see~\cite{Fa}).
\item \label{3poloviny} There exist infinite words having languages closed under reversal, however, containing only a~finite number of palindromes. For an example see~\cite{BeBoCaFa}. For such words, Lemma~\ref{technical_frequency} provides even a~better estimate
$$ \# \{\rho(e)| e \in {\mathcal L}_{n+1} \}\quad
\leq \quad \tfrac{3}{2}\Delta C(n).$$
\item \label{XX} The essential
idea of our approach relies in the fact that the closeness of the
language under reversal implies existence of a~non-triavial
automorphism of the labeled Rauzy graph. More generally, our method
can be applied on any infinite word whose language $\mathcal L$
possesses a~symmetry $T: {\mathcal L} \rightarrow {\mathcal L}$ with
the following properties:
\begin{enumerate}
\item $T$ is a~bijective map,
\item for every $w,v \in {\mathcal L}$,
 $$\#\{\text{occurrences of $w$ in $v$}\}=\#\{\text{occurrences of $T(w)$ in $T(v)$}\}.$$
\end{enumerate}
Clearly, the mirror image map $w \to \overline{w}$ satisfies both
assumptions. A further example can be obtained if we choose
a~permutation $\pi$ of letters and define $T_{\pi}(w_1w_2\dots
w_n)=\pi(w_1)\pi(w_2)\dots \pi(w_n)$ for each factor $w_1w_2\dots
w_n$. It may be shown that the group of all such symmetries $T$ is
generated by the mirror image map and the mappings $T_{\pi}$.
\item If the language of a~binary word is closed under exchange $\pi$ of
letters (such words are called complementation-symmetric), no simple
path is mapped by $\pi$ on itself and, thus, each frequency is
assigned to at least two edges in a~reduced Rauzy graph
$\tilde{\Gamma}_n$. As the number of edges is at most $3\Delta
C(n)$, we obtain for frequencies the same upper bound as in
Item~\eqref{3poloviny}.
\item The
Thue-Morse sequence has in the sense of Item~\eqref{XX} the most
symmetrical language among binary words. It explains why the upper
bound from Theorem~\ref{UpperBoundReversalClosed} overestimates the
actual number of factor frequencies. For concrete values of factor
frequencies consult Frid~\cite{Fr}.
\end{enumerate}
\section{Acknowledgment}
The authors acknowledge the financial support of Czech Science
Foundation GA\v{C}R 201/05/0169 and the Ministry of Education of
the Czech Republic LC06002.


\end{document}